\documentclass[12pt]{article}
\usepackage[utf8]{inputenc}
\usepackage{amsmath}
\usepackage[english]{babel}
\usepackage{url}

\usepackage[a4paper,top=3cm,bottom=2cm,left=3cm,right=3cm,marginparwidth=1.75cm]{geometry}

\usepackage{amsmath}
\usepackage{graphicx}
\usepackage[colorinlistoftodos]{todonotes}
\usepackage[colorlinks=true, allcolors=blue]{hyperref}
\usepackage{amssymb}
\usepackage{amssymb}
\usepackage{amsthm}
\usepackage{authblk}
\usepackage{amsmath}
\usepackage{amsfonts}
\usepackage{amssymb}
\usepackage{amscd}
\usepackage{latexsym}
\usepackage{breqn}
\usepackage{cite}
\usepackage{mathtools}
\mathtoolsset{showonlyrefs}

\theoremstyle{plain}
\newtheorem{thm}{Theorem}[section]
\newtheorem{lem}[thm]{Lemma}
\newtheorem{prop}[thm]{Proposition}
\newtheorem{cor}[thm]{Corollary}

\theoremstyle{definition}

\newtheorem{conj}[thm]{Conjecture}
\newtheorem{question}[thm]{Question}

\theoremstyle{remark}

\title{Semialgebraic Dimension and Truncated Toeplitz Models for Complex Symmetric Matrices}
\author{Ryan O'Loughlin \\ E-mail address: \href{mailto:r.d.oloughlin@reading.ac.uk}{r.d.oloughlin@reading.ac.uk}}
\affil{Department of Mathematics and Statistics, University of Reading, 
Reading, United Kingdom, RG6 6AX}
\date{}
\begin{document}

\maketitle
\begin{abstract}
We answer negatively a finite-dimensional unitary-model question for complex symmetric operators. More precisely, we show that, for every \(n\geq 10\), not every
\(n\times n\) symmetric matrix is unitarily equivalent to a direct sum of truncated
Toeplitz operators. In order to do this, we first use semialgebraic dimension, a tool from real algebraic geometry, to prove
a general theorem showing that, if \(\mathcal X\) is a semialgebraic family of
complex symmetric matrices, then the set of complex symmetric matrices which are unitarily
equivalent to an element of \(\mathcal X\) is semialgebraic and has dimension at most $\dim_{\mathbb R}\mathcal X+\frac{n(n-1)}2.$ We then apply this theorem to show that when $n\geq 10$ there exist irreducible symmetric $n \times n$ matrices which are not unitarily equivalent to a truncated Toeplitz operator. Although unitary equivalence is too restrictive, we prove that every
finite-dimensional complex symmetric operator is complex-orthogonally
equivalent to a coanalytic truncated Toeplitz operator. We also answer
positively a refined representation question by showing that whenever a
symmetric matrix is unitarily equivalent to a truncated Toeplitz operator,
it is a matrix representation of that operator with respect to a
conjugation-invariant orthonormal basis.

 \vskip 0.5cm
\noindent Keywords: complex symmetric operator, truncated Toeplitz operator, Toeplitz matrix, model space.
 \vskip 0.5cm
\noindent MSC: 30H10, 47B35, 15B05, 46E20.
\end{abstract}

\section{Introduction}

Since the seminal work of Sarason on truncated Toeplitz operators \cite{sarason2007algebraic}, there has been a substantial amount of interest in the structure and applications of truncated Toeplitz operators. In parallel with this, there has also been a growing body of work on complex symmetric operators. These two classes of operators are closely connected, since every truncated Toeplitz operator is complex symmetric, and truncated Toeplitz operators have often served as prototypical examples in the theory of complex symmetric operators.

There is an emerging body of evidence to suggest that truncated Toeplitz operators may serve as model operators for complex symmetric operators. In the open-problems paper \cite{garciaopen}, the following question was posed as one of three highlighted questions in the theory of complex symmetric operators.

\begin{question}\label{Q}
Just as multiplication operators $M_z : L^2(X, \mu ) \to L^2(X, \mu)$ play
a fundamental role in decomposing normal operators, can one develop a comparable model theory for complex symmetric operators? There is some indication that truncated Toeplitz operators may play an important role in the resolution of this problem.
\end{question}

Decomposing normal operators, namely the Spectral Theorem, is a famous cornerstone result of Operator Theory. A positive answer to the above question would therefore represent a significant step towards a model theory for a large class of non-normal operators. However, this question is not yet fully understood even in finite dimensions. In finite dimensions, complex symmetric operators are precisely those operators which are unitarily equivalent to a symmetric matrix. In this regard, the following question from \cite{recentprogressGarcia} has received particular attention.

\begin{question}\label{conj}
    Every symmetric matrix is unitarily equivalent to a direct sum of truncated Toeplitz operators.
\end{question}

The question above is known to be true for normal matrices \cite{recentprogressGarcia}, $2\times 2$, matrices \cite{recentprogressGarcia}, $3\times 3$ matrices \cite{garciaUETTOanalyticsymbols}, rank one matrices \cite{garciaUETTOanalyticsymbols}, nilpotent matrices of order 2 \cite{garcia2014two} and for certain tensor products of truncated Toeplitz operators \cite{UETTOstrouse}. Related structural results appear in \cite{TTOUEsimilarity}, where the authors
study spatial isomorphism, unitary equivalence and similarity for truncated Toeplitz
operators. It was also shown in \cite{garcia2012unitary} that every finite complex symmetric
	matrix is unitarily equivalent to a direct sum of (i) irreducible complex symmetric matrices
	or (ii) matrices of the form $A \oplus A^T$ where $A$ is irreducible
	and not unitarily equivalent to a complex symmetric matrix.

The principal aim of this paper is to resolve Question \ref{conj} in the negative. More precisely, for every \(n\geq 10\), we show not every
$n \times n$ symmetric matrix is unitarily equivalent to a direct sum of truncated
Toeplitz operators. A central novelty of the proof is that it brings tools from real algebraic geometry
into this part of operator theory. Using semialgebraic dimension, we prove a general
dimension-counting result for unitary equivalence classes inside the space of complex symmetric matrices.

After establishing this negative result, we turn to a related constructive question
from earlier work. Question \ref{conj} asks whether every symmetric matrix is
unitarily equivalent to a direct sum of truncated Toeplitz operators, and is therefore an existence question.
By contrast, one may also ask whether a unitary equivalence which already exists
between a symmetric matrix and a truncated Toeplitz operator can be realised through a natural choice of
basis on the model space. Following the negative answer in \cite{o2023symmetric} to a question
posed in \cite{recentprogressGarcia}, a refined version of this problem was proposed,
asking whether every unitary equivalence between a complex symmetric matrix and a
truncated Toeplitz operator can be realised by a matrix representation with respect
to a conjugation-invariant orthonormal basis. We show that this refined question has a positive answer in finite dimensions.

The layout of the paper is as follows. In Section \ref{2} we introduce the basic material
on semialgebraic sets and semialgebraic dimension which will be used throughout
the paper. In Section \ref{3} we prove a general dimension-counting theorem for unitary
equivalence classes inside the space of complex symmetric matrices. In Section \ref{4}
we apply this theorem to finite-dimensional truncated Toeplitz operators and show
that Question \ref{conj} has a negative answer for \(n\geq 10\). In Section \ref{new5} we introduce complex-orthogonal equivalence and prove that
every finite-dimensional complex symmetric operator has a coanalytic
truncated Toeplitz model under this equivalence.  Finally, in
Section \ref{5} we prove every unitary equivalence between a complex symmetric matrix and a
truncated Toeplitz operator can be realised by a matrix representation with respect
to a conjugation-invariant orthonormal basis.

\section{Semialgebraic preliminaries}\label{2}

In this section we introduce the basic material on semialgebraic sets needed in the sequel. Our aim is only to introduce the definitions and dimension facts used in the proof, and to make the argument accessible to readers who are not specialists in real algebraic geometry. All of the material recalled below is standard; for a more detailed treatment of semialgebraic sets, semialgebraic dimension, and the results used here, we refer the reader to \cite{bochnak2013real}.

Although many of the conditions which throughout the paper are algebraic, it is essential
for us to work in the context of semialgebraic sets rather than only with algebraic
varieties. The sets we need involve inequalities, such as the distinctness of points
on the unit circle, the non-vanishing of certain matrix entries, and the
non-triviality of reducing projections. We also pass repeatedly to complements and projections. Semialgebraic sets are stable under these operations, by the Tarski--Seidenberg theorem, and they still carry a robust dimension theory. This is
what makes semialgebraic dimension the natural tool for the argument.

We denote the set of $n \times n$ symmetric matrices with entries in $\mathbb{C}$ by $\mathcal{CS}_n$ and the set of $ n \times n $ real orthogonal (i.e. real unitary) matrices by $O(n,\mathbb R)$. Throughout the paper, we identify \(M_n(\mathbb C)\) with \(\mathbb R^{2n^2}\) by taking real and imaginary parts of the matrix entries. Similarly \(\mathcal{CS}_n\) is identified with \(\mathbb R^{n(n+1)}\).

A subset \(X\subseteq\mathbb R^N\) is called semialgebraic if it can be written as a finite union of sets of the form
\[
\{x\in\mathbb R^N:p_1(x)=0,\ldots,p_r(x)=0,\ q_1(x)>0,\ldots,q_s(x)>0\},
\]
where \(p_i,q_j\in\mathbb R[x_1,\ldots,x_N]\). Equivalently, semialgebraic sets are precisely the finite Boolean combinations of polynomial equalities and inequalities. In particular, complements of semialgebraic sets are semialgebraic.

If \(X\subseteq\mathbb R^N\) and \(Y\subseteq\mathbb R^M\) are semialgebraic sets, a map \(f:X\to Y\) is called semialgebraic if its graph $$\Gamma_f=\{(x,y)\in X\times Y:y=f(x)\}
$$
is a semialgebraic subset of \(\mathbb R^{N+M}\).  Polynomial maps are semialgebraic, and rational maps are semialgebraic on the subset where their denominators are non-zero.

We shall use the standard notion of real semialgebraic dimension. By the semialgebraic stratification theorem, every semialgebraic set \(X\subseteq\mathbb R^N\) admits a finite disjoint decomposition $$
X=X_1\sqcup\cdots\sqcup X_k ,
$$
where each \(X_j\) is a smooth manifold and itself a semialgebraic subset of $\mathbb{R}^N$. The real dimension of \(X\) is then defined by \(\dim_{\mathbb R}X=\max_{1\leq j\leq k}\dim_{\mathbb R}X_j\). This number is independent of the chosen stratification. We note that when $X$ is a $\mathbb{R}$-vector space, the vector space dimension agrees with the real dimension of \(X\).

We shall use the following standard facts about semialgebraic dimension.

% Although the following lemma contains standard material from real algebraic geometry, we include it to make the argument more accessible to readers coming from complex symmetric operators and truncated Toeplitz operators communities.

\begin{lem}\label{facts}
Let \(X,Y\) be semialgebraic sets.
\begin{enumerate}
\item If \(X\subseteq Y\), then \(\dim_{\mathbb R}X\leq \dim_{\mathbb R}Y\).
\item If \(X=X_1\cup\cdots\cup X_k\), then \(\dim_{\mathbb R}X=\max_j\dim_{\mathbb R}X_j\).
\item If \(f:X\to Y\) is semialgebraic, then \(\dim_{\mathbb R}f(X)\leq \dim_{\mathbb R}X\).  In particular, if two semialgebraic sets are related by a real linear isomorphism, then they have the same real semialgebraic dimension.
\item If \(X\subseteq\mathbb R^a\times\mathbb R^b\) is semialgebraic, then its projection onto either coordinate factor is semialgebraic and has dimension at most \(\dim_{\mathbb R}X\).
\item If \(X\subseteq Y\times\mathbb R^m\) is semialgebraic and every fibre \(X_y=\{z:(y,z)\in X\}\) has dimension at most \(d\), then \(\dim_{\mathbb R}X\leq \dim_{\mathbb R}Y+d\).
\item If \(X\) and \(Y\) are semialgebraic sets, then $X\times Y$ is semialgebraic and 
\[
\dim_{\mathbb R}(X\times Y)
=
\dim_{\mathbb R}X+\dim_{\mathbb R}Y.
\]
\end{enumerate}
\end{lem}

These are standard consequences of semialgebraic stratification and the rank theorem. See, for example, Bochnak--Coste--Roy \cite[Chapter 2]{bochnak2013real} for further details. The assertion in 4 that projections of semialgebraic sets are semialgebraic is the Tarski--Seidenberg theorem.

% If a semialgebraic set is contained in the image of a semialgebraic map from a subset of \(\mathbb R^d\), then its real dimension is at most \(d\). This recovers the intuitive understanding that the number of parameters bounds the dimension. Consequently, the previous lemma shows that if a semialgebraic set is contained in a finite union of images of semialgebraic maps from subsets of \(\mathbb R^d\), then its real dimension is at most \(d\).

% By part (iii) of the previous lemma, if a semialgebraic set is contained in the image of a semialgebraic map from a subset of \(\mathbb R^d\), then its real dimension is at most \(d\). Thus the number of real parameters bounds the dimension. By part (ii) of the previous lemma, the same conclusion holds for a finite union of such images.

We will also need the following observation later.

\begin{lem}\label{lem22}
Let \(V\) be a finite-dimensional real vector space and let \(U\subseteq V\) be a non-empty open semialgebraic subset. Then \(\dim_{\mathbb R}U=\dim_{\mathbb R}V\).
\end{lem}

\begin{proof}
Since \(U\) is non-empty and open, it contains an open ball in \(V\). After choosing a basis for \(V\), this open ball is identified with an open ball in \(\mathbb R^{\dim_{\mathbb R}V}\), which is a smooth semialgebraic manifold of dimension \(\dim_{\mathbb R}V\). Hence \(\dim_{\mathbb R}U\geq \dim_{\mathbb R}V\). The reverse inequality follows immediately, since \(U\subseteq V\). 
\end{proof}

\section{Dimension counting for symmetric unitary equivalence classes of complex symmetric matrices}\label{3}

Let \(\mathcal X\subseteq \mathcal{CS}_n\). We define the symmetric unitary equivalence class of \(\mathcal X\) by
\[
\mathcal U(\mathcal X)
=
\{S\in\mathcal{CS}_n:S\text{ is unitarily equivalent to some }M\in\mathcal X\}.
\]

\noindent Observe that the symmetric unitary equivalence class is the intersection of the full unitary equivalence class with the set of symmetric matrices.

\begin{thm}\label{mainfinal}
Let \(\mathcal X\subseteq\mathcal{CS}_n\) be semialgebraic. Then \(\mathcal U(\mathcal X)\) is semialgebraic and
\[
\dim_{\mathbb R}\mathcal U(\mathcal X)
\leq
\dim_{\mathbb R}\mathcal X+\frac{n(n-1)}2.
\]
\end{thm}

We first prove the lemma which allows us to replace unitary equivalence by real orthogonal equivalence.

\begin{lem}\label{3.2}
Let \(M\in\mathcal{CS}_n\). Then
\[
\{S\in\mathcal{CS}_n:S\text{ is unitarily equivalent to }M\}
=
\{R^TMR:R\in O(n,\mathbb R)\}.
\]
\end{lem}

\begin{proof}
The inclusion from right to left is immediate. Conversely, suppose \(S\in \mathcal{CS}_n\) is such that $S=V^*MV$ for some unitary \(V\). Since \(S=S^T\), we have
\[
V^*MV=V^TM \overline{V}.
\]
Multiplying on the left by \(V\) and on the right by \(V^T\) gives
\[
MVV^T=VV^TM.
\]
Set \(Q=VV^T\). Then \(Q\) is symmetric, unitary, and commutes with \(M\). Choose a symmetric unitary square root \(W\) of \(Q\), commuting with \(M\). This can be done by taking \(W=p(Q)\), where \(p\) is chosen by interpolation on the finite spectrum of \(Q\) so that \(p(\lambda)^2=\lambda\). Then \(W^2=Q\), \(W^T=W\), and \(WM=MW\).

Now set $ R=W^{-1}V. $ Then
\[
RR^T
=
W^{-1}VV^TW^{-T}
=
W^{-1}QW^{-1}
=
I.
\]
Also as \(R\) is unitary, \(R^T = R^{-1}=R^*\), and hence \(R\) is real, so \(R\in O(n,\mathbb R)\). Finally,
\[
S=V^*MV=R^TW^*MWR=R^TMR.
\]
So every complex symmetric representative in the unitary orbit of \(M\) lies in the real orthogonal orbit
\[
\{R^TMR:R\in O(n,\mathbb R)\}.
\]

\end{proof}

In the proof of Theorem \ref{mainfinal} we use the standard dimension count
\begin{equation}\label{dimcount}
    \dim_{\mathbb R}O(n,\mathbb R)=\frac{n(n-1)}2.
\end{equation}
This can be seen since the columns of a matrix in \(O(n,\mathbb R)\) form an orthonormal basis
of \(\mathbb R^n\). Choosing these columns successively gives
\[
(n-1)+(n-2)+\cdots+1=\frac{n(n-1)}2
\]
real parameters.

\begin{proof}[Proof of Theorem \ref{mainfinal}]
By the previous lemma,
\[
\mathcal U(\mathcal X)
=
\{R^TMR:M\in\mathcal X,\ R\in O(n,\mathbb R)\}.
\]
Consider the map
\[
\Phi:\mathcal X\times O(n,\mathbb R)\to\mathcal{CS}_n,
\qquad
\Phi(M,R)=R^TMR.
\]
The map \(\Phi\) is polynomial in the entries of \(M\) and \(R\), and hence is semialgebraic. Since \(\mathcal U(\mathcal X)=\Phi(\mathcal X\times O(n,\mathbb R))\),  Lemma \ref{facts} (4) implies that \(\mathcal U(\mathcal X)\) is semialgebraic.

Moreover, by Lemma \ref{facts} (3) semialgebraic maps do not increase dimension, so
\[
\dim_{\mathbb R}\mathcal U(\mathcal X)
\leq
\dim_{\mathbb R}\bigl(\mathcal X\times O(n,\mathbb R)\bigr).
\]
Lemma \ref{facts} (6) and the dimension count equation \eqref{dimcount} now give
\[
\dim_{\mathbb R}\bigl(\mathcal X\times O(n,\mathbb R)\bigr)
=
\dim_{\mathbb R}\mathcal X+\frac{n(n-1)}2,
\]
and so
\[
\dim_{\mathbb R}\mathcal U(\mathcal X)
\leq
\dim_{\mathbb R}\mathcal X+\frac{n(n-1)}2.
\]
\end{proof}

\section{Direct sums of truncated Toeplitz operators}\label{4}

\subsection{Truncated Toeplitz operator preliminaries}

We denote the open unit disc in the complex plane by $\mathbb{D}$ and the unit circle by $\mathbb{T}$. The Hardy space, $H^2$, is the space of all analytic functions $f(z) = \sum_{n=0}^{\infty}a_n z^n $ on $\mathbb{D}$ such that 
$$
\|f\|=\left(\sum_{n=0}^{\infty}\left|a_n\right|^2\right)^{\frac{1}{2}}<\infty .
$$
We refer the reader to \cite{duren1970theory, nikolski2002operators} for a detailed background on the Hardy space. Let 
\[
B(z)=\lambda\prod_{j=1}^n \frac{z-a_j}{1-\overline{a_j}z},
\qquad |\lambda|=1,\quad a_j\in\mathbb D.
\]
be a finite Blaschke product of degree \(n\), and let
\[
K_B=H^2\ominus BH^2
\]
be the corresponding model space so that \(\dim K_B=n\). We refer the reader to \cite{modelspacesgarcia, cima2000backward} for a detailed background on model spaces.

For $\phi \in L^{\infty}(\mathbb{T})$, we define the truncated Toeplitz operator (which we abbreviate to TTO), $A_{\phi}:  K_{B} \to K_{B}$, by 
$$
A_{\phi}(f) = P_{B} ( \phi f ),
$$
where $P_{B}:L^2(\mathbb{T}) \to K_{B}$ is the orthogonal projection. For the Blaschke product $B(z) =  z^n$, we have  $K_{z^n} = \operatorname{span} 1,z,z^2, ..., z^{n-1}$ and every TTO on $K_{z^n}$ is a Toeplitz matrix, so TTOs may be viewed as generalisations of Toeplitz matrices.

The following theorem is an immediate consequence of Theorem 1.15 in \cite{TTOonfinitedimn}. The original result gives a more explicit representation, including the precise unitary equivalence and the constants appearing in the formula. For our purposes, only the resulting structural form is needed.

\begin{thm}\label{Ross}
Let \(B\) be a finite Blaschke product of degree \(n\), let \(\alpha\in\mathbb T\), and write
\[
B^{-1}(\alpha)=\{\zeta_1,\ldots,\zeta_n\}.
\]
Let \(A\) be a TTO on \(K_B\). Then \(A\) is unitarily equivalent to a matrix
\(M_A=(m_{i,j})\in\mathcal{CS}_n\) for which there exist vectors \(x,y\in\mathbb C^n\) such that
\begin{equation}\label{1.16}
(\zeta_i-\zeta_j)m_{i,j}=x_i y_j-y_i x_j
\end{equation}
for all \(1\leq i,j\leq n\).
\end{thm}

\textcolor{red}{}

\subsection{Showing Question \ref{conj} is false}

In this section we use the previous dimension bounds to show Question \ref{conj} is false. Before discussing an outline to show Question \ref{conj} is false, we prove the following lemma. In the following lemma $[X,Y] = XY-YX$ is the commutator of $X,Y \in M_{n}(\mathbb{C})$.

\begin{lem}\label{ranklem}
For $M_A$ as given by \eqref{1.16} and $D(B, \alpha) = \operatorname{diag}( \zeta_1, \zeta_2, \ldots , \zeta_n)$, where $\zeta_i \in B^{-1}(\alpha)$ we have 
$$
\operatorname{rank} ([ D(B, \alpha), M_A ]) \leq 2.
$$
\end{lem}
\begin{proof}
For every \(i,j\),
\[
[D(B, \alpha),M_A]_{i,j}=(\zeta_i-\zeta_j)m_{i,j}.
\]
By Theorem \ref{Ross}, there exist vectors \(x,y\in\mathbb C^n\) such that
\[
[D(B, \alpha),M_A]_{i,j}=x_i y_j-y_i x_j.
\]
Therefore
\[
[D(B, \alpha),M_A]=xy^T-yx^T
\]
is the sum of two rank-one matrices, and hence
\[
\operatorname{rank}[D(B, \alpha),M_A]\leq2.
\]
\end{proof}

We denote the set of matrices 
\begin{align}
&\mathcal{I} = \{ \text{irreducible } n \times n \text{ matrices} \} \\
&\mathcal{D}
=
\left\{
\operatorname{diag}(\zeta_1,\ldots,\zeta_n):
\zeta_i\in\mathbb T,\ 
\zeta_i\neq \zeta_j \text{ for } i\neq j
\right\}, \\
&\mathcal{C}
=
\left\{
M\in\mathcal{CS}_n:
\exists D\in\mathcal{D}
\text{ such that }
\operatorname{rank}[D,M]\leq 2
\right\}, \\
\end{align}

% Observe that if Question \ref{conj} were true, then every irreducible complex symmetric matrix would be unitarily equivalent to a single TTO. Since the basis appearing in Theorem \ref{Ross} is orthonormal, the previous lemma would then imply that every irreducible matrix in \(\mathcal{CS}_n\) belongs to \(\mathcal{U} ( \mathcal{C})\). Hence Question \ref{conj} would imply
% \[
% \mathcal{I}\cap\mathcal{CS}_n=\mathcal{I}\cap\mathcal{U} ( \mathcal{C}).
% \]
% We shall disprove Question \ref{conj} by proving the following theorem.

\begin{thm}\label{mainthm}
Let $n \geq 10$. Then $\mathcal{U} ( \mathcal{C})$ and $(\mathcal{I} \bigcap \mathcal{CS}_n)$ are semialgebraic and
$$
\dim_{\mathbb{R}} \mathcal{U} ( \mathcal{C}) < \dim_{\mathbb{R}} (\mathcal{I} \bigcap \mathcal{CS}_n).
$$
\end{thm}

\begin{cor}\label{counterex}
    Let $n \geq  10$. Then Question \ref{conj} is false.
\end{cor}
\begin{proof}
Suppose for contradiction Question \ref{conj} is true. Then every irreducible matrix in \(\mathcal{CS}_n\) is unitarily equivalent to a single TTO. Since \(B\) is a finite Blaschke product and \(\alpha\in\mathbb T\), the points in
\(B^{-1}(\alpha)=\{\zeta_1,\ldots,\zeta_n\}\) lie in \(\mathbb T\) and are pairwise distinct, and hence
\[
D(B,\alpha)=\operatorname{diag}(\zeta_1,\ldots,\zeta_n)\in\mathcal D.
\]
So from Lemma \ref{ranklem} we deduce every irreducible matrix in \(\mathcal{CS}_n\) belongs to \(\mathcal{U} ( \mathcal{C})\). Thus we conclude
\[
\mathcal{I}\cap\mathcal{CS}_n \subseteq \mathcal{I}\cap\mathcal{U} ( \mathcal{C}),
\]
and consequently 
$$
\mathcal{I}\cap\mathcal{CS}_n \subseteq \mathcal{U} ( \mathcal{C}).
$$
This contradicts the previous theorem by Lemma \ref{facts} (1).
\end{proof}

\subsection{Proof of Theorem \ref{mainthm}}

\begin{lem}
Let \(D \in \mathcal{D} \). Then
\[
\dim_{\mathbb R}\{M\in \mathcal{CS}_n:\operatorname{rank}[D,M]\leq 2\}\leq 6n-6.
\]
\end{lem}

\begin{proof}
Let \(M=(m_{ij})\in \mathcal{CS}_n\). It is readily checked that
\[
[D,M]_{ij}=(\zeta_i-\zeta_j)m_{ij},
\]
and thus $[D,M]$ is skew-symmetric. Furthermore, the map
\[
M\mapsto \bigl((m_{11},\ldots,m_{nn}),[D,M]\bigr)
\]
is a real linear isomorphism from \(\mathcal{CS}_n\) onto
\[
\mathbb C^n\times\{C\in M_n(\mathbb C):C^T=-C\},
\]
as its inverse is
$$
((d_1,\ldots,d_n),C)\mapsto M=(m_{ij}),\qquad
m_{ii}=d_i,\quad m_{ij}=\frac{C_{ij}}{\zeta_i-\zeta_j}\quad (i\neq j).
$$

Under this real linear isomorphism, the subset
\[
\{M\in\mathcal{CS}_n:\operatorname{rank}[D,M]\leq2\}
\]
is mapped bijectively onto
\[
\mathbb C^n\times
\{C\in M_n(\mathbb C):C^T=-C,\ \operatorname{rank}C\leq2\}.
\]
Since the ambient map and its inverse are real linear, this restriction is a semialgebraic bijection with semialgebraic inverse. Therefore, by Lemma \ref{facts}(3), the two sets have the same real semialgebraic dimension.
Therefore
\[
\dim_{\mathbb R}\{M\in\mathcal{CS}_n:\operatorname{rank}[D,M]\leq2\}
=
2n+
\dim_{\mathbb R}\{C\in M_n(\mathbb C):C^T=-C,\ \operatorname{rank}C\leq2\}.
\]
Thus to complete the proof we show the real semialgebraic dimension of
\begin{equation}\label{est2}
    \{C\in M_n(\mathbb C):C^T=-C,\ \operatorname{rank}C\leq2\}
\end{equation}
is at most \(4n-6\).

% Under this isomorphism, the condition \(\operatorname{rank}[D,M]\leq2\) becomes the condition that the skew-symmetric matrix \(C=[D,M]\) has rank at most \(2\). Thus it remains to estimate the real semialgebraic dimension of
% \begin{equation}\label{est2}
%     \{C\in M_n(\mathbb C):C^T=-C,\ \operatorname{rank}C\leq2\}.
% \end{equation}

For any four distinct indices \(p,q,i,j\), consider the corresponding \(4\times4\) principal submatrix $C_{p,q,i,j}$ of \(C\), obtained by restricting to the rows and columns indexed by \(p,q,i,j\). Since \(C\) is skew-symmetric, this submatrix is skew-symmetric and direct computation of the determinant of $4 \times 4$ skew symmetric matrices yields
$$
    \det (C_{p,q,i,j} ) = \bigl(C_{pq}C_{ij}-C_{pi}C_{qj}+C_{pj}C_{qi}\bigr)^2.
$$
Since \(\operatorname{rank} C\leq 2\), this determinant must vanish. Hence
\begin{equation}\label{deteqn}
C_{pq}C_{ij}-C_{pi}C_{qj}+C_{pj}C_{qi}=0.
\end{equation}

% For any four distinct indices \(p,q,i,j\), consider the corresponding \(4\times 4\) principal submatrix of \(C\), obtained by restricting \(C\) to the rows and columns indexed by \(p,q,i,j\). Since \(C\) is skew-symmetric, this principal submatrix is also skew-symmetric. Since \(\operatorname{rank}C\leq 2\), every \(4\times4\) principal submatrix of \(C\) has determinant zero. For a \(4\times4\) skew-symmetric matrix, the determinant is the square of its Pfaffian. Hence the Pfaffian of every \(4\times4\) principal skew-symmetric submatrix of \(C\) vanishes. Therefore, for indices \(p,q,i,j\),
% \[
% C_{pq}C_{ij}-C_{pi}C_{qj}+C_{pj}C_{qi}=0.
% \] Therefore, for indices \(p,q,i,j\),
% \[
% C_{pq}C_{ij}-C_{pi}C_{qj}+C_{pj}C_{qi}=0.
% \]
% Thus, on the chart \(C_{pq}\neq0\), the whole matrix \(C\) is determined by the entries
% \[
% C_{pq},\qquad C_{pk},C_{qk}\quad(k\neq p,q).
% \]
% There are
% \[
% 1+2(n-2)=2n-3
% \]
% complex parameters. Hence this chart has real dimension at most $2(2n-3)=4n-6$.

% Clearly, there are only finitely many choices of \(p<q\), so the whole rank-\(\leq2\) skew-symmetric set has real dimension at most \(4n-6\).

Fix \(p<q\), and consider the subset
\[
\mathcal S_{pq}
=
\{C:C^T=-C,\ \operatorname{rank}C\leq 2,\ C_{pq}\neq0\}.
\]
On this subset, the entries
\[
C_{pq},\qquad C_{pk},\ C_{qk}\quad (k\neq p,q)
\]
determine all the remaining entries of \(C\) since if \(i,j\notin\{p,q\}\),
then \eqref{deteqn} gives
\[
C_{ij}
=
\frac{C_{pi}C_{qj}-C_{pj}C_{qi}}{C_{pq}}.
\]
The diagonal entries of matrices in $\mathcal{S}_{pq}$ are zero, and the lower triangular entries are determined by skew-symmetry. Thus, \(\mathcal S_{pq}\) is contained in the image of a semialgebraic map from
\(\mathbb C^*\times\mathbb C^{2n-4}\). By Lemma \ref{facts}(3) and Lemma \ref{facts}(6),
\[
\dim_{\mathbb R}\mathcal S_{pq}\leq 2(2n-3)=4n-6.
\]

Finally, the rank-\(\leq2\) skew-symmetric set given as equation \eqref{est2} is the union of the zero matrix and the finitely many subsets \(\mathcal S_{pq}\), \(p<q\). Therefore, by the finite-union property given as Lemma \ref{facts}(2), the set \eqref{est2} has real dimension at most \(4n-6\).

% Returning to \(M\), the diagonal entries \(m_{11},\ldots,m_{nn}\) contribute \(n\) complex parameters, i.e. \(2n\) real parameters. Therefore
% \[
% \dim_{\mathbb R}\{M\in \mathcal{CS}_n:\operatorname{rank}[D,M]\leq2\}
% \leq
% 2n+(4n-6)
% =
% 6n-6.
% \]
\end{proof}

\begin{prop}\label{main1}
The set $\mathcal{C}$ is semialgebraic and $ \dim_{\mathbb{R}}  ( \mathcal{C}) \leq 7n - 6 $.
\end{prop}

\begin{proof}
Set
\[
\mathcal E
=
\{(D,M)\in\mathcal{D}\times \mathcal{CS}_n:\operatorname{rank}[D,M]\leq2\}.
\]
The set \(\mathcal D\) is semialgebraic, since it is defined by the conditions
\[
|\zeta_i|^2=1,\qquad |\zeta_i-\zeta_j|^2>0\quad(i\neq j).
\]
The entries of \([D,M]\) are polynomial functions of the real and imaginary parts
of the entries of \(D\) and \(M\). The condition
\(\operatorname{rank}[D,M]\leq 2\) is equivalent to the vanishing of all
\(3\times 3\) minors of \([D,M]\). Since these minors are complex-valued
polynomials, this is equivalent to the vanishing of their real and imaginary
parts, which are real polynomial equations. Hence \(E\) is semialgebraic.

For each fixed \(D\in\mathcal{D}\), the fibre over \(D\) is
\[
\{M\in \mathcal{CS}_n:\operatorname{rank}[D,M]\leq2\}.
\]
By the previous lemma, this fibre has real dimension at most $6n-6$.
Furthermore as \(\mathcal{D}\) is parametrised by \(n\) distinct points of \(\mathbb T\), we clearly have $\dim_{\mathbb R}\mathcal{D}\leq n$. Hence, by part 5 of Lemma \ref{facts},
\[
\dim_{\mathbb R}\mathcal E
\leq
n+(6n-6)
=
7n-6.
\]
Since \(\mathcal C\) is the projection of \(\mathcal E\) onto the \(M\)-coordinate, Lemma \ref{facts} (4) shows $\mathcal{C}$ is semialgebraic and
\[
\dim_{\mathbb R}\mathcal C\leq7n-6.
\]
\end{proof}

% In particular we have 
% \begin{equation}\label{antisymmetric}
%     \operatorname{rank}( UDU^*SUD^*U^* - S) = \operatorname{rank}([D, U^*SU])
% \end{equation}
% writing $V^* =  UDU^* $ if we can show there exists a unitary matrix with simple spectrum $V$ such that $\operatorname{rank}V^*SV - S \geq 3$ we disprove the conjecture

\begin{thm}\label{main2}
The set $\mathcal{I} \bigcap \mathcal{CS}_n$ is semialgebraic and $\dim_{\mathbb{R}} (\mathcal{I} \bigcap \mathcal{CS}_n) = n(n+1)$
\end{thm}

\begin{proof}
We first show that taking the semialgebraic dimension of \(\mathcal I\cap\mathcal{CS}_n\) is well-defined, by showing that \(\mathcal I\cap\mathcal{CS}_n\) is semialgebraic.

A matrix \(S\in\mathcal{CS}_n\) is reducible if and only if there exists a non-trivial orthogonal projection \(P\) such that \(PS=SP\). The conditions \(P=P^*\), \(P^2=P\), \(PS=SP\), and \(1\leq \operatorname{tr}P\leq n-1\) are polynomial equalities and inequalities in the real and imaginary parts of the entries of \(P\) and \(S\). The inequalities \(1\leq \operatorname{tr}P\leq n-1\) are exactly the condition that \(P\) is non-trivial, since the trace of an orthogonal projection is its rank.

Therefore the set of pairs \((S,P)\) satisfying these conditions is semialgebraic. By Lemma \ref{facts}(4) its projection onto the \(S\)-coordinate is semialgebraic. Hence the set of reducible matrices in \(\mathcal{CS}_n\) is semialgebraic. Therefore its complement, \(\mathcal I\cap\mathcal{CS}_n\), is also semialgebraic.

We next show that the reducible matrices in \(\mathcal{CS}_n\) form a closed set, where \(\mathcal{CS}_n\) is equipped with the Euclidean topology coming from the identification of $\mathcal{CS}_n$ with $\mathbb{R}^{n(n+1)}$. Suppose \(S_k\to S\) in \(\mathcal{CS}_n\), and each \(S_k\) is reducible. Then for each \(k\), there is a non-trivial orthogonal projection \(P_k\) such that \(P_kS_k=S_kP_k\). Passing to a subsequence, we may assume the \(P_k\) all have the same rank. The set $$
\{P:P=P^*=P^2,\ 1\leq \operatorname{tr}P\leq n-1\}
$$ is compact in finite dimensions, so after passing to a subsequence, \(P_k\to P\) for some non-trivial orthogonal projection \(P\). Taking limits in \(P_kS_k=S_kP_k\), we get \(PS=SP\). Thus \(S\) is reducible. 

The set \(\mathcal I\cap\mathcal{CS}_n\) is clearly non-empty. Hence we have shown \(\mathcal I\cap\mathcal{CS}_n\) is a non-empty open semialgebraic subset of \(\mathcal{CS}_n\). Since \(\dim_{\mathbb R}\mathcal{CS}_n=n(n+1)\) lemma \ref{lem22} yields
\[
\dim_{\mathbb R}(\mathcal I\cap\mathcal{CS}_n)=n(n+1).
\]
\end{proof}

\begin{proof}[Proof of Theorem \ref{mainthm}]
By Proposition \ref{main1}, the set \(\mathcal C\) is semialgebraic. Hence
Theorem \ref{mainfinal} implies that \(\mathcal U(\mathcal C)\) is semialgebraic.
By Theorem \ref{main2}, the set \(\mathcal I\cap\mathcal{CS}_n\) is also
semialgebraic. It remains to prove the dimension inequality.

Proposition \ref{main1} and Theorem \ref{mainfinal} show that
$$
\dim_{\mathbb{R}} \mathcal{U} ( \mathcal{C}) \leq 7n - 6 + \frac{n(n-1)}{2}.
$$
As $n \geq 10$, we have $7n - 6 + \frac{n(n-1)}{2} < n(n+1)$ and so combining this fact with the above estimate and the previous theorem yields
$$
\dim_{\mathbb{R}} \mathcal{U} ( \mathcal{C}) \leq 7n - 6 + \frac{n(n-1)}{2} < n(n+1) = \dim_{\mathbb R}(\mathcal I\cap\mathcal{CS}_n).
$$
\end{proof}

\section{A complex-orthogonal truncated Toeplitz model}\label{new5}

In \cite{garciaopen}, Garcia asked whether truncated Toeplitz operators
could provide a model theory for complex symmetric operators.
Corollary~4.4 shows that unitary equivalence, which preserves the usual
Hilbert-space inner product, is too restrictive for a universal
finite-dimensional truncated Toeplitz model. We instead use
complex-orthogonal equivalence, which preserves the symmetric bilinear
form induced by a conjugation. Theorem~\ref{answergarcia} shows that every
finite-dimensional complex symmetric operator has a coanalytic truncated
Toeplitz model under this notion of equivalence, giving a modest positive
contribution towards Garcia's question.

If $\mathcal{H}$ is a Hilbert space, then we say that $C$ is a conjugation operator on $\mathcal{H}$ if the following conditions hold:
\newline 
(a) $C$ is antilinear:
$$
C\left(a_{1} f_{1}+a_{2} f_{2}\right)=\overline{a_{1}} C f_{1}+\overline{a_{2}} C f_{2}
$$
for all $a_{1}, a_{2} \in \mathbb{C}$ and $f_{1}, f_{2} \in \mathcal{H}$.\newline
(b) $C$ is isometric:
\begin{equation}\label{b}
\langle C f, C g\rangle=\langle g, f\rangle
\end{equation}
for all $f, g \in \mathcal{H}$.
\newline
(c) $C$ is involutive: \begin{equation}\label{c}
C^{2}=I .
\end{equation}
\newline An operator $F$ on a Hilbert space $\mathcal{H}$ is said to be \textit{$C$-symmetric} if $CFC = F^*$, and an operator is \textit{complex symmetric} if it is $C$-symmetric with respect to some conjugation map $C$. Associated with each conjugation $C$ is the symmetric bilinear form
\[
        [x,y]_C:=\langle x,Cy\rangle,
        \qquad x,y\in\mathcal H.
\]

Let $\mathcal H_1$ and $\mathcal H_2$ carry conjugations $C_1$ and $C_2$, respectively. For $X\colon\mathcal H_1\to\mathcal H_2$, define its $(C_1,C_2)$-transpose by \[ X^t=C_1X^*C_2. \] Thus \[ [Xx,y]_{C_2}=[x,X^ty]_{C_1}. \] We call an invertible map $X\colon\mathcal H_1\to\mathcal H_2$ $(C_1,C_2)$-orthogonal if \[ [Xx,Xy]_{C_2}=[x,y]_{C_1} \] for all $x,y\in\mathcal H_1$, and say that operators $T_j$ on $(\mathcal H_j,C_j)$ are complex-orthogonally equivalent if $XT_1=T_2X$ for such a map $X$. This is the natural two-space extension of the standard notion of a $C$-orthogonal operator \cite{garciaapp}.

A vector $x\in\mathcal H$ is called $C$-real if $Cx=x$. Every
finite-dimensional Hilbert space with conjugation $C$ admits an
orthonormal basis of $C$-real vectors (see Lemma~2.6 in \cite{CCOgarcia}),
and such a basis is called a $C$-real orthonormal basis. If $X$ is
$(C_1,C_2)$-orthogonal, then, with respect to $C_1$-real and $C_2$-real
orthonormal bases, its matrix $Q$ satisfies $Q^TQ=I$. Thus, on
$\mathbb C^n$ with standard coordinate conjugation
$C_0z=\overline z$, this is precisely the usual notion of
complex-orthogonal equivalence. For a finite Blaschke product $B$, the model space $K_B$ carries the canonical conjugation \[ (C_Bf)(\zeta)=B(\zeta)\overline{\zeta f(\zeta)}, \qquad \zeta\in\mathbb T, \] and it is known that every truncated Toeplitz operator on $K_B$ is $C_B$-symmetric \cite{sarason2007algebraic}. 

\begin{thm}\label{answergarcia} Let $T$ be a $C$-symmetric operator on an $n$-dimensional Hilbert space $\mathcal H$. Then there exist a finite Blaschke product $B$ of degree $n$, a function $\varphi\in H^\infty$, and a $(C,C_B)$-orthogonal isomorphism $O\colon\mathcal H\to K_B$ such that \[ OTO^{-1}=A_{\overline{\varphi}}. \] \end{thm}

\begin{proof} By \cite[Theorem~6.1]{TTOUEsimilarity}, there exist a finite Blaschke product $B$ of degree $n$, a coanalytic truncated Toeplitz operator $A=A_{\overline{\varphi}}$ on $K_B$, and an invertible map $Y\colon\mathcal H\to K_B$ such that \[ YT=AY. \] Since $T^t=T$ and $A^t=A$, taking transposes gives \[ TY^t=Y^tA. \] Set
\[
        G=Y^tY.
\]
Then $G$ is invertible, $G^t=G$, and
\[
        GT=Y^tYT=Y^tAY=TY^tY=TG.
\]
We now use the same square-root normalization as in the proof of
Lemma~3.2, with the transpose determined by $C$ and $C_B$. By Hermite
interpolation, there is a polynomial $p$ such that $W=p(G)$ satisfies
\[
        W^2=G.
\]
Since $G^t=G$ and $GT=TG$, we also have
\[
        W^t=W,\qquad WT=TW.
\]
Define $O=YW^{-1}$. Then
\[
        O^tO=W^{-1}Y^tYW^{-1}=W^{-1}GW^{-1}=I,
\]
so $O$ is $(C,C_B)$-orthogonal, and
\[
        OT=YW^{-1}T=YTW^{-1}=AYW^{-1}=AO.
\] \end{proof}

\section{Conjugation-invariant matrix representations of truncated Toeplitz operators}\label{5}

The following conjecture was posed in \cite{o2023symmetric} after a related
question involving modified Clark bases was shown to have a negative answer.

\begin{conj}\label{newconj}
Suppose that \(M\) is a complex symmetric matrix. If \(M\) is unitarily equivalent
to a TTO does there exist an inner function \(B\), and a \(C_B\)-real orthonormal basis for the
model space \(K_B\), such that \(M\) is the matrix representation of a TTO on
\(K_B\) with respect to this basis? In other words, do all such unitary
equivalences between complex symmetric matrices and TTOs arise from \(C_B\)-real
matrix representations?
\end{conj}
\noindent In this section we prove this conjecture in the finite-dimensional setting.

% We finish by recording a consequence of Lemma \ref{3.2}.

Recall that, for a finite Blaschke product \(B\), the model space \(K_B\) carries the canonical
conjugation
\[
(C_B f)(\zeta)=B(\zeta)\overline{\zeta f(\zeta)},\qquad \zeta\in\mathbb T.
\]
The following theorem proves Conjecture \ref{newconj}.

\begin{thm}
Let \(B\) be a finite Blaschke product of degree \(n\), let \(A\) be a truncated
Toeplitz operator on \(K_B\), and let \(M\in\mathcal{CS}_n\). If \(M\) is
unitarily equivalent to \(A\), then there exists a \(C_B\)-real orthonormal basis
\(v_1,\ldots,v_n\) of \(K_B\) such that
\[
M=[A]_{v_1,\ldots,v_n}.
\]
\end{thm}

\begin{proof}
Choose a \(C_B\)-real orthonormal basis \(b_1,\ldots,b_n\) of \(K_B\), and set
\[
S=[A]_{b_1,\ldots,b_n}.
\]
Since \(A\) is \(C_B\)-symmetric and \(b_1,\ldots,b_n\) is a \(C_B\)-real basis, we have \(S \in\mathcal{CS}_n\).

By assumption, \(M\) is unitarily equivalent to \(A\). Since \(S\) is a matrix
representation of \(A\), it follows that \(M\) and \(S\) are unitarily equivalent
symmetric matrices. By Lemma \ref{3.2}, there exists \(R=(r_{ij})\in O(n,\mathbb R)\)
such that
\[
M=R^TSR.
\]
Define
\[
v_j=\sum_{i=1}^n r_{ij}b_i,\qquad 1\leq j\leq n.
\]
Because \(R\) is real orthogonal, \(v_1,\ldots,v_n\) is an orthonormal basis.
Moreover, since the coefficients \(r_{ij}\) are real and each \(b_i\) is
\(C_B\)-real, each \(v_j\) is also \(C_B\)-real. Finally, the change-of-basis
formula gives
\[
[A]_{v_1,\ldots,v_n}=R^TSR=M.
\]
\end{proof}

Comparing the theorem above to Corollary \ref{counterex} shows that while not every symmetric matrix is unitarily equivalent to a direct sum of
TTOs, any unitary equivalence which does exist between a symmetric matrix and a truncated Toeplitz operator can be realised as
a \(C_B\)-real matrix representation.

\newpage 

\newpage

\newpage 

\bibliographystyle{plain}
\bibliography{bibliography.bib}

\end{document}